\documentclass{amsproc}

\usepackage{amsmath,amssymb,amsthm,mathtools}
\usepackage{amsrefs}
\usepackage[hidelinks]{hyperref}
\usepackage{tikz}

\newtheorem{theorem}{Theorem}[section]
\newtheorem{corollary}[theorem]{Corollary}
\newtheorem{proposition}[theorem]{Proposition}
\newtheorem{lemma}[theorem]{Lemma}

\newtheorem{definition}[theorem]{Definition}

\newtheorem*{conjecture}{Conjecture}
\newtheorem{question}[theorem]{Question}
\newtheorem*{theoremA}{Theorem A}
\newtheorem*{theoremB}{Theorem B}

\newcommand{\rk}{\operatorname{rk}}
\newcommand{\vol}{\operatorname{vol}}

\newcommand{\alg}{\operatorname{alg}}
\newcommand{\Prob}{\mathbb P}
\newcommand{\E}{\mathbb E}

\makeatletter
\@namedef{subjclassname@2020}{2020 Mathematics Subject Classification}
\makeatother

\author{Andreas Thom}
\address{Andreas Thom, TU Dresden, 01062 Dresden, Germany}
\email{andreas.thom@tu-dresden.de}

\begin{document}

\title{On prime endomorphisms of the free group of rank two}


\begin{abstract}
We study prime endomorphisms of the free group $F_2$. The main results provide lower and upper bounds for the logarithmic density $\delta_{\mathrm{np}}$ of non-prime endomorphisms:
$$
\frac34 \leq       \delta_{\mathrm{np}}\leq \frac{27}{28}.
$$
Equivalently, we prove corresponding lower and upper bounds for the exponential decay of the probability that a random endomorphism is non-prime, aligning with a heuristic argument of Ian Agol.
\end{abstract}

\maketitle

\tableofcontents

\section*{Introduction}

Every endomorphism of the additive group $\mathbb Z$ has the form
$x\mapsto mx$ for some $m\in\mathbb Z$. The prime number theorem says that the number of primes less than or equal to $n$ is approximately equal to $\frac{n}{\log n}$. Thus the prime number theorem may be viewed as a counting theorem for endomorphisms of $\mathbb Z$ that cannot be written as a composition of non-bijective endomorphisms.

This suggests the analogous counting problem for the free group $F_2$. First of all, we say that an injective endomorphism of $F_2$ is prime if any way of writing it as a composition of endomorphisms must involve an automorphism, see Definition~\ref{def:prime}. Hence, the question is: How often is an endomorphism of $F_2 = F(a,b)$ prime, when one orders endomorphisms by the lengths of the images of a fixed free basis $(a,b)$?  Writing $\varphi_{u,v} \colon F_2 \to F_2$ for the
endomorphism with $a\mapsto u$ and $b\mapsto v$, the natural finite model is to
choose $(u,v)$ uniformly from $B_n\times B_n$, where $B_n$ is the ball of
radius $n$ in $F_2$.  The problem was suggested by Ian Agol, who also outlined
in \cite{ThomMO} a heuristic argument for the fact that the density of non-prime endomorphisms should tend to zero. Our result is more precise than this, giving explicit lower and upper bounds of the exponential decay rate for the density of non-prime endomorphisms.

Since $|B_n|=2\cdot 3^n-1$, the total number of pairs in $B_n\times B_n$ has
order about $\Theta(3^{2n})$.  Therefore, it is natural to measure the size of exceptional sets
by its logarithmic density.  Let
$
        \mathcal N_n=
        \{(u,v)\in B_n\times B_n:\varphi_{u,v}\text{ is not prime}\}
$
and define
\[
        \delta_{\mathrm{np}}=
        \limsup_{n\to\infty}\frac{\log |\mathcal N_n|}{\log |B_n\times B_n|}.
\]
Because
$
        \log |B_n\times B_n|=2n\log 3+O(1),
$
this quantity can equally be recovered from the exponential decay of probability that a random endomorphism is not prime.
In particular, an upper bound of this probability of the form $n^B3^{-\alpha n}$ implies
$
        \delta_{\mathrm{np}}\leq 1-\frac{\alpha}{2}.
$

\medskip

Our two main results are about lower and upper bounds for $\delta_{\mathrm{np}}$.

\begin{theoremA}
Let $B_n$ be the ball of radius $n$ in $F_2$. There are absolute constants
$A,B>0$ such that
\[
        \Prob_{(u,v)\in B_n\times B_n}
        \bigl(\varphi_{u,v}\text{ is not prime}\bigr)
        \leq A n^B3^{-n/14}.
\]
Consequently, we obtain
$
        \delta_{\mathrm{np}}\leq \frac{27}{28}.
$
\end{theoremA}

The proof of the upper bound has three ingredients. First, primeness is
translated into the absence of a proper intermediate rank-two subgroup. Second, we restrict our attention to pairs $(u,v)$ that satisfy a small cancellation condition and show that the generated subgroup cannot lie properly inside a rank-two overgroup whose Stallings core has large volume. Third, proper rank-two subgroups have a
sharp uniform cogrowth gap: their non-backtracking spectral radius is at most
$\rho_0<9/4$, where $\rho_0$ is the largest real root of
$\rho^3-\rho^2-\rho-3=0$.  Combining this with an explicit count of small volume
rank-two cores gives the bound $n^B3^{-n/14}$.

\begin{theoremB}

$
        \delta_{\mathrm{np}}\geq \frac34.
$
\end{theoremB}

The lower bound comes from a high-entropy family of variable endomorphic images
obtained by factoring through $\varphi_{a^2,b}$.
Thus, as a consequence of the main results, the logarithmic density of non-prime pairs is constrained by the inequality
\[
        \frac34
        \leq \delta_{\mathrm{np}}
        \leq \frac{27}{28}=0.96429 \dots.
\]
It is a natural problem to determine the exact value of $\delta_{\mathrm{np}}$. We conjecture that the lower bound is sharp.

\begin{conjecture}
$
        \delta_{\mathrm{np}}=\frac34.
$
\end{conjecture}

The paper is organized as follows. Section~\ref{S.basic} collects basic facts
about injective endomorphisms and rank-two overgroup lattices,
Sections~\ref{S.clean}--\ref{S.random-words} prove the sphere estimate,
Section~\ref{S.log-density} turns it into logarithmic-density bounds in the ball
model, and Section~\ref{S.random-walk} treats simple random walks. We end with a section of questions and remarks.

\section{Basic observations on endomorphisms}\label{S.basic}

Let $F_2=F(a,b)$ be the free group of rank two.  For $u,v\in F_2$, write
$H(u,v)=\langle u,v\rangle \leq F_2$ and let $\varphi_{u,v}\colon F_2\to F_2$ denote
the endomorphism with $a\mapsto u$ and $b\mapsto v$.

\subsection{Prime endomorphisms and algebraic overgroups}

We start out by collecting some basic facts about endomorphisms of $F_2$.

\begin{lemma}\label{L.injective.noncommuting}
The endomorphism $\varphi_{u,v}$ is injective if and only if $u$ and $v$ do not
commute.
\end{lemma}

\begin{proof}
If $u$ and $v$ commute, then $H(u,v)$ is cyclic, so $\varphi_{u,v}$ cannot be
injective.  Conversely, if $u$ and $v$ do not commute, then $H(u,v)$ is a
non-cyclic subgroup of the free group $F_2$ generated by two elements.  Hence
$H(u,v)$ is free of rank two, and one can show that the images of $a$ and $b$ freely generate the
image of $\varphi_{u,v}$.  Therefore $\varphi_{u,v}$ is injective.
\end{proof}

\begin{definition} \label{def:prime}
An endomorphism $\varphi\colon F_2\to F_2$ is called \emph{prime} if it is
injective and, whenever $\varphi=\alpha\circ\beta$ with
$\alpha,\beta\colon F_2\to F_2$  endomorphisms, at least one of
$\alpha$ or $\beta$ is an automorphism.
\end{definition}

The following elementary criterion is the bridge between factorization of
endomorphisms and subgroup theory.

\begin{lemma}\label{L.primecriterion}
Let $\varphi\colon F_2\to F_2$ be an endomorphism and put $H=\varphi(F_2)$.
Then $\varphi$ is prime if and only if $\varphi$ is injective and there is no
subgroup $K$ with $H<K<F_2$ and $\rk K=2$.
\end{lemma}

\begin{proof}
Suppose first that $\varphi$ is prime.  Then $\varphi$ is injective by
definition.  If there were a subgroup $K$ with $H<K<F_2$ and $\rk K=2$,
choose an isomorphism $\alpha\colon F_2\to K$.  Since $H\leq K$, the map
$\beta=\alpha^{-1}\circ\varphi\colon F_2\to F_2$ is an injective
endomorphism and $\varphi=\alpha\circ\beta$.  The subgroup $K$ is proper in
$F_2$, so $\alpha$ is not an automorphism.  Also $H<K$, so
$\beta(F_2)<F_2$, and hence $\beta$ is not an automorphism.  This contradicts
primeness.  Therefore no such subgroup $K$ exists.

Conversely, suppose that $\varphi$ is injective and that there is no subgroup
$K$ with $H<K<F_2$ and $\rk K=2$.  If $\varphi=\alpha\circ\beta$ with
$\alpha,\beta\colon F_2\to F_2$ endomorphisms and neither an automorphism, then
the injectivity of $\varphi$ implies that $\beta$ is injective.  Moreover,
$\alpha(F_2)$ contains $H=\varphi(F_2)$, which has rank two because $\varphi$ is
injective.  Hence $\alpha(F_2)$ is non-cyclic, and therefore $\alpha$ is
injective by Lemma~\ref{L.injective.noncommuting}.  It follows that
$H=\varphi(F_2)=\alpha(\beta(F_2))<\alpha(F_2)<F_2$.  Since $\alpha$ is
injective, $\alpha(F_2)$ is free of rank two.  Thus $\alpha(F_2)$ is a proper
intermediate rank-two subgroup, contradiction.  Therefore $\varphi$ is prime.
\end{proof}

The following definition has been studied extensively in \cite{MVW}.

\begin{definition} \label{def:alg}
Let $H\leq J$ be free groups.  We say that $J$ is an
\emph{algebraic extension} of $H$, and write $H\leq_{\alg}J$, if $H$ is not
contained in a proper free factor of $J$.
\end{definition}

\begin{lemma}
Let $H\leq K\leq F_2$, $H$ non-cyclic and $\rk K=2$.  Then
$H\leq_{\alg}K$.
\end{lemma}

\begin{proof}
If $H$ were contained in a proper free factor $L$ of $K$, then $\rk L\leq 1$
because $K$ has rank two.  This is impossible since $H$ is non-cyclic.
\end{proof}

Thus, for a rank-two subgroup $H\leq F_2$, non-primality of the corresponding
injective endomorphism is equivalent to the existence of a proper rank-two
algebraic extension $H\leq_{\alg}K<F_2$.

\subsection{Rank-two factorization lattices}

For a finitely generated non-cyclic subgroup $H\leq F_2$, let
$P_2(H)=\{K\leq F_2: H\leq K,\ \rk K=2\}$.  We regard $P_2(H)$ as a partially
ordered set by inclusion and consider it to be an analogue of the prime factorization of a natural number. It is implied by the study of algebraic extension in free groups \cite{MVW}, that this poset is finite, see below for details. Moreover, as a consequence of the solution to the Hanna-Neumann conjecture \cites{Mineyev,Friedman},
we will prove below that it is a lattice; but unlike for $\mathbb Z$, it is not distributive and in fact not even modular in general.

\medskip

By the preceding lemma, every element of $P_2(H)$ is algebraic over $H$.
Finiteness of algebraic extensions for finitely generated subgroups of free
groups \cite{MVW} therefore gives the following consequence.

\begin{corollary}
If $H\leq F_2$ is finitely generated and non-cyclic, then $P_2(H)$ is finite. 
\end{corollary}

\begin{corollary}
Every injective endomorphism of $F_2$ admits a factorization into prime
endomorphisms. Moreover, modulo post- and pre-composition of adjacent factors by inverse automorphisms,
there are only finitely many such factorizations.
\end{corollary}

\begin{proof}
Let $\varphi\colon F_2\to F_2$ be injective and put $H=\varphi(F_2)$.  If
$\varphi$ is not prime, Lemma~\ref{L.primecriterion} gives a proper
intermediate rank-two subgroup $H<K<F_2$.  Choosing an isomorphism
$\alpha\colon F_2\to K$ gives a nontrivial factorization
$\varphi=\alpha\circ(\alpha^{-1}\circ\varphi)$.  Iterating this procedure
corresponds to passing upward in the finite poset $P_2(H)$, so the process
terminates.  Maximal chains in $P_2(H)$ therefore give prime factorizations.
Conversely, any prime factorization determines such a chain, up to the evident
freedom of inserting inverse automorphisms between adjacent factors.  Since
$P_2(H)$ is finite, only finitely many chains, and hence only finitely many
such factorizations, occur.
\end{proof}

\begin{proposition} \label{P.meet}
Let $H\leq F_2$ be non-cyclic and let $K,L\in P_2(H)$.  Then
$K\cap L\in P_2(H)$.  In particular, the meet of $K$ and $L$ in $P_2(H)$ is
their ordinary intersection.
\end{proposition}

\begin{proof}
Using the  Hanna-Neumann inequality \cite{Mineyev},
$\max\{\rk(K\cap L)-1,0\}\leq (\rk K-1)(\rk L-1)=1$, so $\rk(K\cap L)\leq 2$.
On the other hand, $H\leq K\cap L$, and since $H$ is non-cyclic, the subgroup
$K\cap L$ is non-cyclic.  Hence $\rk(K\cap L)=2$.
\end{proof}

The join operation is not the ordinary subgroup join, because $\langle K,L\rangle$
may have rank larger than two.  The correct substitute is the minimal rank-two
overgroup.

\begin{proposition}
Let $G\leq F_2$ be finitely generated and non-cyclic.  Then there is a unique
minimal rank-two subgroup of $F_2$ containing $G$.
\end{proposition}

\begin{proof}
The set $P_2(G)$ is nonempty because $F_2\in P_2(G)$, and it is finite by the
preceding corollary.  If $P_2(G)=\{K_1,\dots,K_m\}$, then repeated application
of the Proposition \ref{P.meet} shows that $K_1\cap\cdots\cap K_m$ still has rank two
and contains $G$.  It is therefore the unique minimal element of $P_2(G)$.
\end{proof}

\begin{definition}
For a finitely generated non-cyclic subgroup $G\leq F_2$, let
$\operatorname{cl}_2(G)$ denote the unique minimal rank-two overgroup of $G$.
We call $\operatorname{cl}_2(G)$ the \emph{rank-two closure} of $G$.
\end{definition}

\begin{theorem}
Let $H\leq F_2$ be finitely generated and non-cyclic.  Then $P_2(H)$ is a
finite lattice.  Its meet and join are given by
$K\wedge L=K\cap L$ and $K\vee L=\operatorname{cl}_2(\langle K,L\rangle)$.
\end{theorem}

\begin{proof}
Finiteness and the meet formula were proved above.  For the join, let
$K,L\in P_2(H)$ and set $G=\langle K,L\rangle$.  Then $G$ is finitely generated
and non-cyclic, so $\operatorname{cl}_2(G)$ exists.  It contains $K$ and $L$,
and it is contained in every rank-two subgroup containing both.  Hence it is
the least upper bound of $K$ and $L$ in $P_2(H)$.
\end{proof}

\begin{corollary}
Every finitely generated subgroup $G\leq F_2$ of rank at least two is
contained in a unique minimal rank-two subgroup of $F_2$.
\end{corollary}

\subsection{Examples of rank-two factorization lattices}
\label{sec:examples}

The lattices $P_2(H)$ can already be quite varied.  The next two examples are
not used later, but they indicate the range of possible behaviour. We use the standard Stallings graph convention, see \cite{Stallings} for details. 
For convenience, we recall the standard facts about folded labelled core graphs of subgroups of
free groups. Let \(F(a,b)\) be the free group on the finite basis \(a,b\). An
\(\{a,b\}\)-labelled graph is a graph whose oriented edges are labelled by elements of
\(\{a,a^{-1},b,b^{-1}\}\), with opposite orientations carrying inverse labels. A connected
pointed \(\{a,b\}\)-labelled graph \((\Gamma,*)\) represents the subgroup
\[
\pi_1(\Gamma,*)=\{\,w\in F(a,b): w \text{ labels a reduced closed path at } *\,\}.
\]
It is called folded if, at every vertex, there is at most one outgoing edge with
a given label, equivalently the natural map to the rose with petals \(\{a,b\}\) is
locally injective. It is called a core graph if every edge lies on a reduced
closed path.

For every finitely generated subgroup \(H\leq F(a,b)\), there is a finite
connected folded pointed labelled graph \(\Gamma(H)\), unique up to pointed
label-preserving isomorphism, such that
\[
w\in H
\quad\Longleftrightarrow\quad
w \text{ labels a reduced closed path in } \Gamma(H) \text{ based at } * .
\]
It is obtained by taking a bouquet of loops labelled by a finite generating set
of \(H\), repeatedly folding pairs of equally labelled edges with the same
initial vertex, and then deleting all non-core hanging trees. Conversely, every
finite connected folded pointed labelled core graph represents a finitely
generated subgroup of \(F(a,b)\).

If \(H\neq \{1\}\), then the rank of \(H\) is the first Betti number of its core:
$
\operatorname{rk}(H)=|E(\Gamma(H))|-|V(\Gamma(H))|+1,
$
where \(E(\Gamma(H))\) denotes the set of unoriented edges.

Algebraic overgroups of \(H\) correspond to quotients of \(\Gamma(H)\) obtained by identifying vertices and edges in a way that respects the labels and directions, and then folding.  In particular, rank-two overgroups of \(H\) correspond to quotients of \(\Gamma(H)\) whose core has first Betti number two.

\begin{proposition}
There exists a non-cyclic subgroup $H\leq F_2$ for which $P_2(H)$ is not
modular.
\end{proposition}

\begin{proof}
Let $H=\langle b^2,\ a^{-1}b^2a\rangle\leq F(a,b)$.  The folded labelled core
graph of $H$ looks as follows:
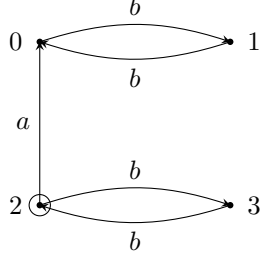
\[
\begin{tikzpicture}[x=1.4cm,y=1.2cm,>=stealth]
        \coordinate (v2) at (0,0);
        \coordinate (v3) at (1.8,0);
        \coordinate (v0) at (0,1.8);
        \coordinate (v1) at (1.8,1.8);

        \fill (v2) circle (1.2pt);
        \draw (v2) circle (4pt);
        \node[left=3pt] at (v2) {$2$};

        \fill (v3) circle (1.2pt);
        \node[right=3pt] at (v3) {$3$};

        \fill (v0) circle (1.2pt);
        \node[left=3pt] at (v0) {$0$};

        \fill (v1) circle (1.2pt);
        \node[right=3pt] at (v1) {$1$};

        \draw[->] (v2) -- node[left] {$a$} (v0);
        \draw[->,bend left=18] (v0) to node[above] {$b$} (v1);
        \draw[->,bend left=18] (v1) to node[below] {$b$} (v0);
        \draw[->,bend left=18] (v2) to node[above] {$b$} (v3);
        \draw[->,bend left=18] (v3) to node[below] {$b$} (v2);
\end{tikzpicture}
\]
The corresponding rank-two overgroups are represented by the following
partitions of the vertex set $\{0,1,2,3\}$:
\[
\begin{array}{cccccccccccc }
T&=&0000,&
A&=&0011,&
B&=&0012,&
C&=&0101,\\
D&=&0110,&
E&=&0122,&
H&=&0123.&
\end{array}
\]
The resulting Hasse diagram is
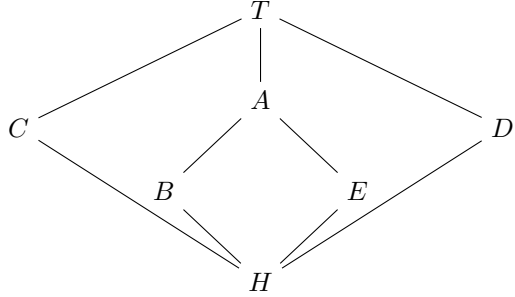
\[
\begin{tikzpicture}[x=1.6cm,y=1.2cm]
        \node (H) at (0,0) {$H$};
        \node (B) at (-0.8,1) {$B$};
        \node (E) at (0.8,1) {$E$};
        \node (C) at (-2,1.7) {$C$};
        \node (D) at (2,1.7) {$D$};
        \node (A) at (0,2) {$A$};
        \node (T) at (0,3) {$T$};

        \draw (H) -- (B) -- (A) -- (T);
        \draw (H) -- (E) -- (A);
        \draw (H) -- (C) -- (T);
        \draw (H) -- (D) -- (T);
\end{tikzpicture}
\]
Thus the Hasse diagram has seven elements.  Taking $x=B$, $y=C$, and $z=A$, we
have $x\leq z$, $y\wedge z=H$, and hence $x\vee(y\wedge z)=B$.  On the other
hand, $x\vee y=T$, so $(x\vee y)\wedge z=A$.  Therefore
$x\vee(y\wedge z)=B\neq A=(x\vee y)\wedge z$, and the modular law fails.
\end{proof}

The previous example shows that there is no uniqueness to the "number" of prime factors of a non-prime endomorphism.  The next example shows that there is also no upper bound on the number of prime factorizations even under reasonable constraints.

\begin{proposition}
For each prime $p$, let $H_p=\langle a^p,\ ba^p b^{-1}\rangle\leq F(a,b)$.
Then $P_2(H_p)$ has height $3$ and cardinality $p+5$.  In particular,
rank-two overgroup lattices of height three have unbounded size.
\end{proposition}

\begin{proof}
The folded core graph of $H_p$ consists of two directed $a$-cycles of length
$p$, joined by one $b$-edge.  Since $p$ is prime, a folded quotient of a
directed $p$-cycle is either again a $p$-cycle or a single vertex with an
$a$-loop.  Moreover, if a vertex of one $a$-cycle is identified with a vertex
of the other and neither cycle is collapsed, then folding forces the two cycles
to be identified with a fixed cyclic shift.  There are exactly $p$ such
diagonal identifications.

The rank-two quotients are therefore as follows: the original graph; the $p$
diagonal quotients; two quotients in which exactly one of the two $a$-cycles is
collapsed to an $a$-loop; one quotient in which both $a$-cycles are collapsed
but the endpoints of the $b$-edge remain distinct; and finally the rose, hence
the full group $F(a,b)$.  Altogether this yields $1+p+2+1+1=p+5$ rank-two
overgroups.

Writing $D_0,\dots,D_{p-1}$ for the diagonal quotients, $L$ and $R$ for the
two quotients in which exactly one $a$-cycle is collapsed, and $M$ for the
quotient in which both $a$-cycles are collapsed but the endpoints of the
$b$-edge remain distinct, the Hasse diagram has the schematic form
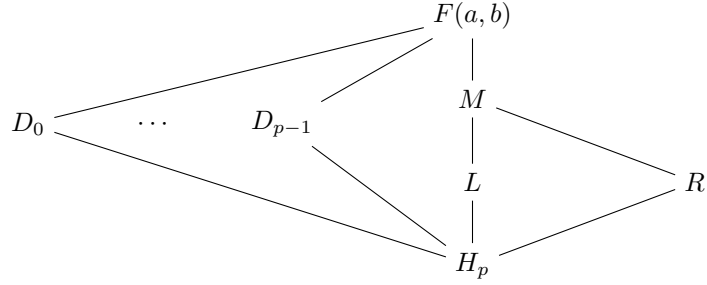
\[
\begin{tikzpicture}[x=1.4cm,y=1.1cm]
        \node (H) at (0,0) {$H_p$};
        \node (L) at (0,1) {$L$};
        \node (R) at (2.1,1) {$R$};
        \node (D0) at (-4.2,1.7) {$D_0$};
        \node (Dots) at (-3,1.7) {$\cdots$};
        \node (Dp) at (-1.8,1.7) {$D_{p-1}$};
        \node (M) at (0,2) {$M$};
        \node (T) at (0,3) {$F(a,b)$};

        \draw (H) -- (L) -- (M) -- (T);
        \draw (H) -- (R) -- (M);
        \draw (H) -- (D0) -- (T);
        \draw (H) -- (Dp) -- (T);
\end{tikzpicture}
\]
In particular, the chain
\[
        H_p<L<M<F(a,b)
\]
has length three, and no longer chain can occur from the description above.
\end{proof}

\begin{question}
\label{Q:lattice}
If $P_2(H)$ has height two, is its cardinality bounded?  More sharply, must
$|P_2(H)|\leq 4$?
\end{question}

\section{Stallings graphs and clean pairs}\label{S.clean}

We now turn to the preparations for the proof of the prime counting theorem.
A reduced word $w \in F_2$ determines a cyclic word after cyclic reduction.  We shall
speak about cyclic occurrences of subwords in the cyclic reductions of
$u,u^{-1},v,v^{-1}$.
The following terminology is a convenient reformulation of the usual
small-cancellation condition; compare \cite{LS}.

\begin{definition}
Let $L\geq 1$.  A pair of reduced words $(u,v)$ is called \emph{$L$-clean} if
the following hold.

\begin{enumerate}
\item The cyclic reductions of $u$ and $v$ have length at least $6L$.

\item In the four cyclic words obtained from the cyclic reductions of \(u^{\pm1}\) and \(v^{\pm1}\), no word of length \(L\) occurs in two distinct cyclic positions.
\end{enumerate}
\end{definition}

\begin{lemma}
Let $(u,v)$ be $L$-clean.  Then
$\langle u,v\rangle$ is freely generated by $u$ and $v$.
\end{lemma}

\begin{proof}
It suffices to show that \(u\) and \(v\) do not commute, since two
noncommuting elements of \(F_2\) generate a free subgroup of rank two, and hence
the given two generators freely generate it.

Suppose that \(u\) and \(v\) commute. Then they lie in a common cyclic subgroup
of \(F_2\). Hence their cyclic reductions are positive powers of a common cyclically reduced word, up to cyclic conjugacy and inversion. This contradicts \(L\)-cleanness. Thus \(u\) and \(v\) do not commute, and the claim follows.
\end{proof}

We shall use the following convention.  For a finitely generated subgroup
\(K\leq F_2\), let \(\widehat\Gamma(K)\) denote its pointed Stallings graph and
let \(\Gamma(K)\) denote its core, obtained from \(\widehat\Gamma(K)\) by deleting
the stem and ignoring the basepoint. Thus \(\Gamma(K)\) is regarded as an unpointed folded labelled core
graph, while \(\widehat\Gamma(K)\) remembers the basepoint. For a finitely generated subgroup \(K\leq F_2\), we write
$
        \operatorname{vol}(K)=|E(\Gamma(K))|
$
for the number of unoriented edges in the core $\Gamma(K)$.
If the stem of \(\widehat\Gamma(K)\) has label \(\gamma\), then conjugating
\(K\) by \(\gamma^{-1}\) moves the basepoint to the point where the stem meets
\(\Gamma(K)\).  The pointed Stallings graph of \(\gamma^{-1}K\gamma\) has no
stem and has the same core as \(K\).  Moreover, for every \(w\in K\), the cyclic
reduction of \(\gamma^{-1}w\gamma\) is cyclically conjugate to the cyclic
reduction of \(w\).  Hence the \(L\)-cleanness condition below is invariant under
such a common conjugation.
In the deterministic arguments of this section we shall therefore freely replace
a pair \(H\leq K\) by a common conjugate, and assume that the basepoint of
\(K\) lies in \(\Gamma(K)\).  This preserves rank, algebraicity, properness of
inclusions, \(L\)-cleanness, and the quantity
$
        \operatorname{vol}(K).
$

The key deterministic input is the following. It says that an \(L\)-clean
rank-two subgroup cannot sit properly inside a rank-two subgroup whose
Stallings core has large volume.

\begin{proposition} \label{P.volume}
Let $(u,v)$ be $L$-clean and put $H=\langle u,v\rangle$.
If $H\leq_{\alg}K$ with $H \not =K$ and $\rk K=2$, then $\vol(K)<3L$.
\end{proposition}

We first record two elementary facts about rank-two core graphs. 

\begin{definition}
Let $X$ be a finite connected core graph of rank two.  A \emph{topological
edge} of $X$ is a maximal edge-path whose internal vertices have degree two and
whose endpoints have degree different from two, with the usual convention for
loops.  Its length is the number of ordinary unoriented edges in that path.
\end{definition}

\begin{lemma}\label{lem:rank-two-core-types}
Let \(X\) be a finite connected core graph with first Betti number \(b_1(X)=2\).
After suppressing degree-two vertices, \(X\) is one of the following three
topological types:
\begin{enumerate}
    \item a figure-eight graph;
    \item a theta graph;
    \item a barbell graph.
\end{enumerate}
In particular, \(X\) has at most three topological edges.
\end{lemma}

\begin{proof}
Let \(\bar X\) be the graph obtained from \(X\) by suppressing degree-two
vertices.  Since \(X\) is a core graph, every vertex of \(\bar X\) has
degree at least three, except in the case of a single vertex with loop-edges.
Since \(b_1(\bar X)=2\), we have
\[
        |E(\bar X)|-|V(\bar X)|+1=2,
\]
and hence \(|E(\bar X)|=|V(\bar X)|+1\).  Therefore
\[
        \sum_{x\in V(\bar X)}(\deg(x)-2)
        =
        2|E(\bar X)|-2|V(\bar X)|
        =
        2.
\]
There are only three possibilities: one vertex with two loop-edges, two vertices
joined by three edges, or two loop-vertices joined by a bridge.  These are,
respectively, the figure-eight, theta, and barbell types.
\end{proof}

\begin{lemma}\label{lem:core-image-fills}
Let \(H\leq_{\mathrm{alg}} K\) with \(\operatorname{rk} H=\operatorname{rk} K=2\).
After conjugating \(H\leq K\), if necessary, assume that the basepoint of the
pointed Stallings graph of \(K\) lies in \(X=\Gamma(K)\).  Let \(Y\subseteq X\)
be the core of the image of the pointed Stallings graph of \(H\) in \(X\).  Then
\(Y=X\).
\end{lemma}

\begin{proof}
Let \(y\in Y\) be the image of the basepoint of the pointed Stallings graph of
\(H\).  The pointed core graph \((Y,y)\) represents an intermediate subgroup
\(J\) with
$
        H\leq J\leq K .
$
Since \(H\) has rank two, \(J\) has rank at least two.  Since \(Y\subseteq X\)
and \(X\) has first Betti number two, \(Y\) has first Betti number at most two.
Thus \(Y\) has first Betti number exactly two.

If \(Y\neq X\), then \(Y\) is a proper connected core subgraph of the connected
core graph \(X\).  Every component of the closure of \(X\setminus Y\) must
contribute a cycle relative to \(Y\); otherwise it would be a tree attached to
\(Y\) at a single vertex and would contain an edge not lying on any reduced
closed path, contrary to the fact that \(X\) is a core graph.  Hence adjoining
\(X\setminus Y\) to \(Y\) strictly increases the first Betti number.  This
contradicts \(b_1(Y)=b_1(X)=2\).  Therefore \(Y=X\).
\end{proof}

\begin{lemma} \label{lem:edge-traversal}
Let $(u,v)$ be $L$-clean and let $X$ be a folded labelled core graph in which
the cyclic reductions of $u$ and $v$ are read as reduced cycles.  Let $e$
be a topological edge of $X$ of length at least $L$.  If the cycles
corresponding to $u$ and $v$ traverse $e$ at least twice in total, counting
both orientations and both cyclic words, then $(u,v)$ is not $L$-clean.
\end{lemma}

\begin{proof}
Choose a subpath $I\subset e$ of length $L$.  Every traversal of $e$ contains a
traversal of $I$, possibly with the opposite orientation.  If two traversals of
$e$ have the
same orientation, then the label of $I$ occurs twice as a cyclic length-$L$
subword among $u$ and $v$.  If the orientations are opposite, then the same
label occurs once in one of $u,v$ and once in one of $u^{-1},v^{-1}$.  In either
case there are two distinct cyclic occurrences of the same reduced word of
length $L$, contradicting $L$-cleanness.
\end{proof}

\begin{lemma}\label{lem:rank-one-periodicity}
Let \(c\) be a nontrivial cyclically reduced word and let \(m\in\mathbb Z\) with
\(|m|\geq 2\).  Let \(w\) be the cyclic word represented by the cyclically
reduced word \(c^m\) if \(m>0\), and by \((c^{-1})^{|m|}\) if \(m<0\).  If
$
        |w|=|m|\,|c|\geq L,
$
then \(w\) contains two distinct cyclic occurrences of the same reduced word of
length \(L\).
\end{lemma}

\begin{proof}
Put \(d=|c|\).  The cyclic word \(w\) is periodic with period \(d\).  Hence the
length-\(L\) cyclic subword starting at a given position depends only on that
position modulo \(d\).

Since \(|w|=|m|d\geq L\), length-\(L\) cyclic subwords are defined.  Because
\(|m|\geq 2\), the cyclic word \(w\) has at least two distinct positions in each
congruence class modulo \(d\).  Choose any cyclic starting position \(i\).  The
starting position \(i+d\), taken modulo \(|w|\), is distinct from \(i\), and by
periodicity the length-\(L\) cyclic subword starting at \(i+d\) is equal to the
one starting at \(i\).  Thus \(w\) has two distinct cyclic occurrences of the
same reduced word of length \(L\).
\end{proof}

\begin{proof}[Proof of Proposition \ref{P.volume}]
Let \(X=\Gamma(K)\).  By the stem-removing reduction above, after replacing
\(H\leq K\) by a common conjugate, we may assume that the basepoint of \(K\)
lies in \(X\).  The elements \(u\) and \(v\) then determine based closed paths
in \(X\).

Let \(Y\subseteq X\) be the core of the image of the pointed Stallings graph
\(\Gamma(H)\) in \(X\).  By Lemma~2.6, we have \(Y=X\).  Thus the based closed
paths labelled by \(u\) and \(v\) fill the whole core graph \(X\).

Suppose that $X$ has a topological edge $e$ of length at least $L$.  We first
observe that $e$ is traversed by at least one of the cyclic reductions of $u$ and
$v$.  Indeed, if not, then every occurrence of $e$ in the based loops labelled by
$u$ and $v$ is contained in the initial-final cancellation of one of these loops.
Since the based loops fill $X$, this can only occur in the barbell case: the edge
$e$ is the bar, and the cyclic reductions of $u$ and $v$ lie in the two rank-one
lobes.  Let $c_1$ and $c_2$ be cyclically reduced generators of the rank-one
subgroups represented by the two lobes.  Thus, up to interchanging $u$ and $v$,
their cyclic reductions are powers $c_1^m$ and $c_2^n$.  By
Lemma~\ref{lem:rank-one-periodicity} and $L$-cleanness, the cyclic reductions of
$u$ and $v$ have length at least $6L$, so $|m| = |n| = 1$.  Since the bar of a
barbell is a separating tree-bridge, it contributes nothing to $\pi_1$, and
$\pi_1(X)$ is freely generated by (conjugates into a common basepoint of) the two
lobe generators $c_1$ and $c_2$.  As $|m| = |n| = 1$, the elements $u$ and $v$ are
conjugate to $c_1^{\pm1}$ and $c_2^{\pm1}$, so they generate all of $\pi_1(X)$.
Hence $H = K$, contrary to $H \ne K$.

Therefore \(e\) is traversed by at least one of the cyclic reductions of \(u\)
and \(v\).  If it is traversed at least twice in total, counting both orientations
and both cyclic words, then Lemma~2.7 contradicts \(L\)-cleanness.  Hence \(e\)
is traversed exactly once in total by the cyclic reductions of \(u\) and \(v\).

If \(e\) is separating, then every cyclic closed path in \(X\) traverses \(e\)
algebraically zero times and, in particular, an even number of times in total.
Thus \(e\) cannot be traversed exactly once by the union of the two cyclic paths.
Hence \(e\) is nonseparating.

Remove the interior of $e$ from $X$ and call the resulting connected graph
$Y=X\setminus \operatorname{int}(e)$. 
Since $X$ has rank two and $e$ is nonseparating, $Y$ has rank one.  Exactly one
of the two cyclic paths, say the path labelled by $u$, traverses $e$ once; the
other path, labelled by $v$, is contained in $Y$.

Let $c$ be a cyclically reduced generator of the rank-one subgroup represented
by the core of $Y$.  Since the cyclic path labelled by $v$ is contained in a
rank-one graph, its cyclic label is a power of $c$: $v \sim_{\mathrm{cyc}} c^m$
for some nonzero integer $m$.
If \(|m|\geq 2\), then Lemma~\ref{lem:rank-one-periodicity} contradicts
\(L\)-cleanness, because the cyclic reduction of \(v\) has length at least
\(6L\).  Therefore \(|m|=1\).

Let $p$ be the initial vertex of the unique traversal of $e$ along the cyclic
path labelled by $u$.  Simultaneously conjugating $H$ and $K$ by the label of a
path from the original basepoint to $p$ preserves all hypotheses and the
strictness of the inclusion.  It replaces $u$ and $v$ by conjugates whose
cyclic reductions are cyclically conjugate to the original ones, so
$L$-cleanness is preserved.  Thus we may rebase $X$ at $p$ and still denote the
resulting subgroup pair by $H\leq_{\alg}K$ and the resulting generators by
$u,v$.

After this rebasing, the loop representing $u$ starts with the unique traversal
of $e$, so $u=er$,
where $r$ is a path in $Y$ from the terminal vertex $q$ of $e$ back to $p$.
Also $p\in Y$, and since the cyclic path labelled by $v$ is contained in $Y$,
the element $v$ lies in $\pi_1(Y,p)$.  Because $\pi_1(Y,p)$ has rank one and the
cyclic label of $v$ is $c^{\pm1}$, the element $v$ generates $\pi_1(Y,p)$.

Choose any path $\beta$ in $Y$ from $p$ to $q$.  Then
$s=e\beta^{-1}$ is the stable loop obtained by adjoining the nonseparating edge
$e$ to $Y$, so $\pi_1(X,p)$ is generated by $\pi_1(Y,p)$ together with $s$.
Since $\beta r$ is a loop in $Y$ based at $p$, there exists $n\in\mathbb Z$
such that $\beta r=v^n$.  Hence $u=er=(e\beta^{-1})(\beta r)=s v^n$.
Therefore $s=u v^{-n}$ lies in $H$.  Since $v$ already generates
$\pi_1(Y,p)$, we obtain $H=\pi_1(X,p)=K$,
contrary to the assumption that $H \not =K$.

Thus no topological edge of $X$ has length at least $L$.  Since a rank-two
core graph has at most three topological edges, we obtain $\vol(K)<3L$.
\end{proof}

\section{Uniform cogrowth gap for proper rank-two subgroups}\label{S.cogrowth}

We now prove the uniform probabilistic input.
Let $S_n$ denote the sphere of radius $n$ in $F_2$ with respect to the basis
$\{a,b\}$.  Thus
$
        |S_n|=4\cdot 3^{n-1}
$
for $n\geq 1$.

If $K\leq F_2$ is finitely generated, then a reduced word belongs to $K$ if and
only if its labelled path in the pointed Stallings graph
$\widehat\Gamma(K)$ returns to the base vertex.  Write $X=\Gamma(K)$ for the
core of $K$, let $h(K)$ be the stem length of $\widehat\Gamma(K)$, let
$x\in X$ be the attachment point of the stem, and let $\gamma$ be the reduced
label of the stem from the basepoint to $x$.  A reduced based loop in
$\widehat\Gamma(K)$ is forced to traverse the stem along $\gamma$ at the
beginning and along $\gamma^{-1}$ at the end, and otherwise stays inside $X$.
Thus $K\cap S_n=\varnothing$ for $n<2h(K)$, and for $n\geq 2h(K)$ every element
of $K\cap S_n$ has the form
$
        \gamma w\gamma^{-1},
$
where $w$ labels a non-backtracking closed path of length $n-2h(K)$ in $X$
based at $x$.

For a finite folded core graph $X$, let $B_X$ be its non-backtracking adjacency
matrix on oriented edges.  Thus
$
        (B_X)_{ef}=1
$
if the terminal vertex of $e$ is the initial vertex of $f$ and
$
        f\neq \bar e,
$
and otherwise
$
        (B_X)_{ef}=0.
$
Let
$
        \rho(X)=\rho(B_X)
$
be its Perron--Frobenius spectral radius.

For the uniform estimates it is more convenient to collapse the degree-two
chains.  Let
$
        \bar X
$
be the graph obtained from $X$ by suppressing degree-two vertices, and let
$\Omega_X$ be the set of its oriented topological edges.  For
$
        e\in \Omega_X
$
let $\ell(e)$ denote the length of the corresponding topological edge in $X$.
Define the weighted transfer matrix
\[
        M_X(z)=\bigl((M_X(z))_{ef}\bigr)_{e,f\in \Omega_X}
\]
by declaring
\[
        (M_X(z))_{ef}=
        \begin{cases}
                z^{\ell(f)} & \text{if }e\to f\text{ is an allowed non-backtracking turn},\\
                0 & \text{otherwise.}
        \end{cases}
\]
Here
$
        e\to f
$
means that the terminal vertex of $e$ is the initial vertex of $f$ and
$
        f\neq \bar e.
$
Thus a topological itinerary
$
        e_1,\dots,e_m
$
contributes the monomial
$
        z^{\ell(e_1)+\cdots+\ell(e_m)}.
$

\begin{lemma}\label{L.sharp-cogrowth}
Let
\[
        \rho_0=
        \frac{1+\sqrt[3]{46+6\sqrt{57}}+\sqrt[3]{46-6\sqrt{57}}}{3}.
\]
For every proper rank-two subgroup $K<F_2$, one has
\[
        \rho(\Gamma(K))\leq \rho_0.
\]
Moreover, equality can occur only when the topological core of $K$ is a
figure-eight with loop lengths $1$ and $2$.
\end{lemma}

\begin{proof}
Let $X=\Gamma(K)$, and suppress degree-two vertices.  By the rank-two
classification lemma, the resulting topological graph is a figure-eight, a
theta graph, or a barbell.

Assume first that $\bar X$ is a figure-eight with loop lengths $a,b$.
After identifying the two orientations of each loop, the weighted $4 \times 4$ transfer
matrix reduces to
\[
        M_{a,b}(r)=
        \begin{pmatrix}
                r^a & 2r^b \\
                2r^a & r^b
        \end{pmatrix}.
\]
The critical value $r=\rho^{-1}$ is determined by
$
        \det\bigl(I-M_{a,b}(r)\bigr)=0,
$
that is,
$
        (1-r^a)(1-r^b)-4r^{a+b}=0.
$
If $a=b=1$, then the core is the rose and represents $F_2$, contrary to
$K<F_2$.  Hence, after relabelling, either $a=1,b\geq 2$, or else $a,b\geq 2$.
Since the entries of $M_{a,b}(r)$ decrease as $a,b$ increase, the maximal proper
figure-eight case is $(a,b)=(1,2)$.  In this case the critical equation is
\[
        (1-r)(1-r^2)-4r^3=0,
\]
or equivalently
$
        1-r-r^2-3r^3=0.
$
Writing $\rho=r^{-1}$, this becomes
$
        \rho^3-\rho^2-\rho-3=0.
$
Thus the maximal figure-eight value is $\rho_0$ using Cardano's formula.

If $\bar X$ is a theta graph, then every oriented ordinary edge has at
most two allowed non-backtracking successors.  Hence every row sum of the
non-backtracking adjacency matrix $B_X$ is at most $2$, so
$
        \rho(X)\leq 2<\rho_0.
$
The same row-sum bound holds for a barbell graph, so again
$
        \rho(X)\leq 2<\rho_0.
$

Therefore $\rho(\Gamma(K))\leq \rho_0$ for every proper rank-two subgroup
$K<F_2$.  Equality can occur only in the figure-eight case with loop lengths
$1$ and $2$.
\end{proof}

\begin{lemma}\label{L.return}
There is a constant $C_0>0$ such that, for every proper rank-two subgroup
$K<F_2$, if $h(K)$ denotes the stem length of $\widehat\Gamma(K)$, then for
every $n\geq 1$,
\[
        \frac{|K\cap S_n|}{|S_n|}
        \leq C_0\left(\frac34\right)^n\left(\frac94\right)^{-2h(K)}.
\]
\end{lemma}

\begin{proof}
Let $X=\Gamma(K)$, let $h=h(K)$, and let $x\in X$ be the attachment point of
the stem of $\widehat\Gamma(K)$.  Let
\[
        F_{X,x}(z)=\sum_{n\geq 1} c_n(X,x)z^n
\]
be the generating series for non-backtracking closed paths in $X$ based at $x$.
By the preceding discussion, the generating series
\[
        F_K(z)=\sum_{n\geq 1}|K\cap S_n|z^n
\]
for reduced words in $K$ satisfies
\[
        F_K(z)=z^{2h}F_{X,x}(z).
\]
We therefore estimate $F_{X,x}(z)$ uniformly.

The point $x$ is either a branch vertex
of $\bar X$ or lies in the interior of a unique topological edge.  In
either case there are only boundedly many possible initial partial segments
leaving $x$ and terminal partial segments returning to $x$.  Consequently
there are vectors $u_{X,x}(z),v_{X,x}(z)\in \mathbb C^{\Omega_X}$ whose entries are
finite sums of at most uniformly many monomials in $z$, and a function $g_X(z)$
such that
\[
        F_{X,x}(z)=g_X(z)+u_{X,x}(z)^\ast\bigl(I-M_X(z)\bigr)^{-1}v_{X,x}(z).
\]
Here $g_X(z)$ is either a polynomial with coefficients bounded by an absolute
constant, or else $x$ lies in the interior of a topological loop of length
$\ell$ and
\[
        g_X(z)=2\sum_{n\geq 1} z^{n\ell}=\frac{2z^\ell}{1-z^\ell}.
\]
Indeed, if $x$ does not lie in the interior of a topological loop, then a
based non-backtracking closed path that never reaches a branch vertex has only
finitely many possibilities, and these are recorded by a polynomial.  If $x$
does lie in the interior of a topological loop of length $\ell$, then the closed
non-backtracking paths that stay entirely in that loop are exactly the nonzero
windings around the loop in one of the two orientations, giving the displayed
geometric series.  All remaining based closed paths are obtained by choosing an
initial partial segment from $x$ to a branch vertex, following a
non-backtracking itinerary of oriented topological edges, and then taking a
terminal partial segment back to $x$.  This is encoded by the displayed
formula.

Set
$
\rho_*=\frac94,
$ and
$       R=\rho_*^{-1}=\frac49.$
By Lemma~\ref{L.sharp-cogrowth}, every proper rank-two core has spectral radius
at most $\rho_0<\rho_*$.  Therefore, at $R=4/9$, the weighted transfer matrices
are uniformly inside their disk of convergence.  More explicitly, for
$|z|\leq R$, the entrywise absolute value of $M_X(z)$ is dominated by one of the
finitely many maximal matrices corresponding to the three topological types: the
figure-eight of lengths $(1,2)$, the theta graph of lengths $(1,1,1)$, or the
barbell graph of lengths $(1,1,1)$.  In the theta and barbell cases the maximal
matrix has row sums at most $2R=8/9<1$.  In the figure-eight case the maximal
matrix is
\[
        \begin{pmatrix}
                R & 2R^2 \\
                2R & R^2
        \end{pmatrix},
\]
whose spectral radius is strictly smaller than $1$ because its critical value is
$\rho_0^{-1}>R$.  Hence the Neumann series
\[
        \bigl(I-M_X(z)\bigr)^{-1}=\sum_{j\geq 0} M_X(z)^j
\]
is uniformly bounded for all proper rank-two $K$ and all $|z|\leq R$.
In the polynomial case $g_X(z)$ is uniformly bounded there, and in the loop case
one has
\[
        |g_X(z)|\leq \frac{2R}{1-R}=\frac{8}{5}.
\]
Since each entry of $u_{X,x}(z)$ and $v_{X,x}(z)$ is a sum of at most uniformly many
monomials $z^d$ and $|z|\leq R<1$, these entries are also uniformly bounded
there,
the generating series $F_{X,x}(z)$ is bounded by an absolute constant on that disk,
independently of $K$.

Therefore $F_K(z)=z^{2h}F_{X,x}(z)$ is bounded by $C_1R^{2h}$ on $|z|\leq R$.
By Cauchy's estimate, the coefficient of $z^n$ in $F_K(z)$ satisfies
\[
        |K\cap S_n|\leq C_1 R^{2h-n}=C_1 \rho_*^{\,n-2h}.
\]
Since $|S_n|=4\cdot 3^{n-1}$, we obtain
$\frac{|K\cap S_n|}{|S_n|}\leq C_0\left(\frac34\right)^n\left(\frac94\right)^{-2h}$.
\end{proof}

\begin{lemma}
        \label{L.count-cores}
There is an absolute constant $A>0$ such that the number of rooted folded
labelled rank-two core graphs over the alphabet $\{a,b\}$ with at most $m$
unoriented edges is at most
$
        A(m+1)^4 3^m.
$
\end{lemma}

\begin{proof}
After suppressing degree-two vertices, the graph has one of three topological
types: figure-eight, theta, or barbell.  Hence it has at most three
topological edges.

For each of the three topological types, the number of choices of the lengths
of the topological edges, subject to total length at most $m$, is bounded by
$(m+1)^3$.  If a topological edge has length $\ell$, then it admits at most
$4\cdot 3^{\ell-1}$ reduced labels.  Since the total length is at most $m$ and
there are at most three topological edges, the total number of labellings is at
most $A_0 3^m$ for an absolute constant $A_0>0$.  Finally, a graph with at most
$m$ unoriented edges has at most $m+1$ vertices, so the root contributes at
most $m+1$ choices.

Multiplying these bounds and absorbing the three topological types into the
constant gives
\[
        \#\{\text{rooted folded labelled rank-two cores with at most }m
        \text{ edges}\}
        \leq A(m+1)^4 3^m
\]
for an absolute constant $A>0$.
\end{proof}

\section{Random reduced words}\label{S.random-words}

We first prove the result for the non-backtracking model, equivalently for
uniform random words on spheres.

Let $u_n,v_n$ be independent uniformly distributed elements of $S_n$ and put
$H_n=\langle u_n,v_n\rangle$.

For a reduced word $w$, write $|w|_{\mathrm{cyc}}$ for the length of its cyclic
reduction.

\begin{lemma}
Let $u_n$ be uniformly distributed on $S_n$.  Then for every integer
$0\leq k\leq \frac{n}{2}$ one has
$\Prob\bigl(|u_n|_{\mathrm{cyc}}\leq n-2k\bigr)\leq C\,3^{-k}$
for an absolute constant $C>0$.

Moreover, if $m\geq 1$ has the same parity as $n$, then conditioned on
$|u_n|_{\mathrm{cyc}}=m$, 
the cyclic reduction of $u_n$ is uniformly distributed among cyclically
reduced words of length $m$.
\end{lemma}

\begin{proof}
If $|u_n|_{\mathrm{cyc}}\leq n-2k$, then $u_n$ has the form $u_n=r z r^{-1}$,
where $r$ is reduced of length $k$ and $z$ is reduced, or trivial, of length
$n-2k$.  Thus the number of such words is at most
\[
        |S_k|\cdot \max\{1,|S_{n-2k}|\}
        \leq C_0\,3^{n-k-1}
\]
for an absolute constant $C_0$.  Since $|S_n|=4\cdot 3^{n-1}$, this gives
$\Prob\bigl(|u_n|_{\mathrm{cyc}}\leq n-2k\bigr)\leq C\,3^{-k}$.

For the second claim, write $m=n-2k$ and fix a cyclically reduced word
$c=c_1\cdots c_m$ 
of length $m$.  The reduced words of length $n$ whose cyclic reduction is $c$
are exactly the words $r c r^{-1}$ 
with $|r|=k$ for which the concatenation is reduced.  If $k=0$, there is only
the word $c$ itself.  If $k\geq 1$, the last letter of $r$ may be any of the
two letters different from $c_1^{-1}$ and $c_m$, and once that last letter is
fixed, the preceding $k-1$ letters of $r$ may be chosen arbitrarily subject
only to reducedness.  Hence there are exactly $2\cdot 3^{k-1}$ such words $r$,
independently of $c$.  Therefore every cyclically reduced word
of length $m$ has the same number of preimages, proving the conditional
uniformity.
\end{proof}

The following lemma is in one form or another well known, but we include a proof for completeness.

\begin{lemma}
Let $C_m$ be a uniformly distributed cyclically reduced word of length $m$.
Fix two distinct cyclic starting positions and consider the corresponding cyclic
subwords of length $L$, either both in $C_m$ or with one taken in $C_m$ and the
other in $C_m^{-1}$, with $m\geq 6L$.

Then there is an absolute constant $C>0$ such that the probability that these
two length-$L$ subwords have the same label is at most $C\,3^{-L/2}$.
\end{lemma}

\begin{proof}
Because the law of $C_m$ is invariant under cyclic rotation, we may place the
first starting position at $1$.  Each chosen occurrence determines an interval
of $L$ consecutive positions in the original cyclic word $C_m$; if an
occurrence is taken in $C_m^{-1}$, we first pull it back to the corresponding
interval in $C_m$ by reversing orientation and inverting labels.  The union of
the two resulting intervals has size at most $2L$, so since $m\geq 6L$ we may
cut the cyclic word at a position outside this union and regard all relevant
letters as lying in a genuine linear reduced word.

First note that a uniformly distributed cyclically reduced word of length $m$ is
obtained from a uniformly distributed reduced word of length $m$ by conditioning
on the last letter not being the inverse of the first.  Given the first $m-1$
letters, the last letter has three possible values, and at most one of them is
forbidden by cyclic reduction.  Hence this conditioning changes probabilities by
at most the factor
$
        3/2.
$
It is therefore enough to prove an exponential bound for a uniform reduced word.

Assume first that both occurrences lie in the same orientation of the word, so
we are comparing two length-$L$ blocks in a uniform reduced word $W$.  Let $t$
be the distance between their starting positions.  If $t\geq L$, then the two
blocks are disjoint, and after revealing the first block the second must equal
one prescribed reduced word of length $L$.  Thus the probability is at most
$3^{-L}$.

If $1\leq t<L$, then equality of the two blocks is equivalent to saying that the
linear segment of length $L+t$ spanning them has period $t$.  Once the first
$t$ letters are chosen, the next $L$ letters are forced.  Hence the number of
possible reduced labels for that segment is at most $4\cdot 3^{t-1}$, whereas
the total number of reduced labels of length $L+t$ is $4\cdot 3^{L+t-1}$.  So
the probability is again at most $3^{-L}$.

It remains to consider the case where one occurrence is taken in $W^{-1}$.  View
that occurrence back in $W$ as the reverse inverse of a length-$L$ block.  If
the two corresponding length-$L$ intervals in $W$ are disjoint, then after
revealing the first block the second interval must realize one prescribed
reduced word, so the probability is at most $3^{-L}$.

Assume finally that these two intervals overlap.  Let $I$ be the minimal linear
segment of $W$ containing their union, and write $s=|I|$.  Then $L\leq s<2L$.
Equality of the original block with the inverse occurrence induces an
orientation-reversing involution on the positions of $I$.  A fixed point would
force a letter to equal its own inverse, which is impossible, so every orbit has
size two.  Hence choosing the letters on one representative from each orbit
determines the whole segment $I$.  Therefore the number of reduced labels on $I$
compatible with the overlap condition is at most
$
        4\cdot 3^{\lceil s/2\rceil-1}.
$
Since the total number of reduced labels of length $s$ is $4\cdot 3^{s-1}$, the
probability of this overlap event is at most
$
        3^{-\lfloor s/2\rfloor}\leq 3^{-L/2}.
$

Thus, for a uniform reduced word, the probability is at most $3^{-L/2}$ in all
cases.  Multiplying by the conditioning factor $3/2$ and absorbing constants
proves the lemma.
\end{proof}

\begin{lemma}
There is an absolute constant $A>0$ such that for every integer $L$ with
$1\leq L<n/6$,
\[
        \Prob\bigl((u_n,v_n)\text{ is not }L\text{-clean}\bigr)
        \leq A\left(3^{-(n-6L)/2}+n^2 3^{-L/2}\right).
\]
\end{lemma}

\begin{proof}
Let $\widehat u_n,\widehat v_n$ 
be the cyclic reductions of $u_n$ and $v_n$.  Applying the previous lemma with
$k=\left\lceil \frac{n-6L}{2}\right\rceil$ gives
\[
        \Prob\bigl(|\widehat u_n|<6L\text{ or }|\widehat v_n|<6L\bigr)
        \leq A_0\,3^{-\lceil (n-6L)/2\rceil}
        \leq A_0\,3^{-(n-6L)/2}.
\]

Now condition on $|\widehat u_n|=m$ and $|\widehat v_n|=\ell$, 
where $m,\ell\geq 6L$.  By the second part of the previous lemma,
$\widehat u_n$ and $\widehat v_n$ are independent uniformly distributed
cyclically reduced words of lengths $m$ and $\ell$.

It remains to estimate repeated cyclic factors of length $L$ in
$\widehat u_n,\widehat u_n^{-1},\widehat v_n,\widehat v_n^{-1}$.  There are at
most $(2m+2\ell)^2\leq 16n^2$ ordered pairs of cyclic starting positions.

For two fixed distinct starting positions in the same cyclic word, the fixed
cyclic block estimate gives
\[
        \Prob\bigl(\text{the two length-}L\text{ blocks agree}\bigr)
        \leq A_1 3^{-L/2}.
\]
If the two starting positions lie in different cyclic words, then after
revealing the first block, the second block must realize one prescribed reduced
word of length $L$, so the same bound holds.  Finally, a block in an inverse
cyclic word is obtained from a block in the original cyclic word by reversing
the order and inverting the letters.  Thus the same bound also applies when one
or both occurrences are taken in inverse cyclic words.

Therefore
\[
        \Prob\bigl(\text{some length-}L\text{ cyclic factor repeats}
        \mid |\widehat u_n|=m,\ |\widehat v_n|=\ell\bigr)
        \leq A_2 n^2 3^{-L/2}.
\]
Averaging over $m,\ell\geq 6L$ and combining with the first estimate gives
\[
        \Prob\bigl((u_n,v_n)\text{ is not }L\text{-clean}\bigr)
        \leq A_0 3^{-(n-6L)/2}+A_2 n^2 3^{-L/2}.
\]
This proves the claim.
\end{proof}

\begin{theorem}\label{T.sphere}\label{thm:sphere-prime}
There are absolute constants $A,B>0$ such that
\[
        \Prob\bigl(\varphi_{u_n,v_n}\text{ is not prime}\bigr)
        \leq An^B 3^{-n/14}.
\]
\end{theorem}

The exponent $1/14$ comes from balancing the two cleanliness errors.  The
cyclic-cancellation term is $3^{-(n-6L)/2}$, while the repeated-block term is
$n^2 3^{-L/2}$.  Choosing $L=\lfloor n/7\rfloor$ balances these two
contributions.  The overgroup contribution is smaller after the improved
cogrowth estimate, since it is bounded by
$
        3^{3L}\left(\frac34\right)^{2n}.
$

\begin{proof}
Put
$
        L=\left\lfloor\frac n7\right\rfloor.
$
Let $\mathcal G_n$ be the event that $(u_n,v_n)$ is $L$-clean.  By the
explicit cleanliness lemma,
\[
        \Prob(\mathcal G_n^c)
        \leq A_0\left(3^{-(n-6L)/2}+n^2 3^{-L/2}\right).
\]
Since $L\leq n/7$ and $L\geq n/7-1$, the right-hand side is at most
$
        A_1 n^2 3^{-n/14}.
$
On $\mathcal G_n$, the pair $(u_n,v_n)$ freely generates a rank-two subgroup.

Assume that $\mathcal G_n$ holds and that there exists
\(H_n<K<F_2\) with $\rk K=2$.  Then $H_n\leq_{\alg}K$
by the rank argument in Section 1.  The deterministic rank-two quotient lemma
therefore gives $\vol(K)<3L$.

Thus, on $\mathcal G_n$, non-primality implies that both $u_n$ and $v_n$ lie in
some proper rank-two subgroup $K<F_2$ with $\vol(K)<3L$.

By the counting lemma, the number of possible rooted folded labelled rank-two
cores of volume less than $3L$ is at most
$
        A_2 (L+1)^4 3^{3L}.
$
For each such rooted core and each integer $h\geq 0$, there are at most
$C_0'3^h$ possible reduced stem labels of length $h$.  Hence the number of
proper rank-two subgroups $K<F_2$ with $\vol(K)<3L$ and
$h(K)=h$ is at most
$
        A_3 (L+1)^4 3^{3L+h}.
$
For each fixed such subgroup, the return estimate gives
\[
        \Prob(u_n\in K,\ v_n\in K)
        \leq C_0^2 \left(\frac34\right)^{2n}\left(\frac94\right)^{-4h}.
\]

Therefore
\[
\begin{aligned}
&\Prob\bigl(\exists K<F_2:\ \rk K=2,\ \vol(K)<3L,
                  \ u_n,v_n\in K\bigr) \\
&\qquad\leq
        A_4 (L+1)^4 3^{3L} \left(\frac34\right)^{2n}
        \sum_{h\geq 0} 3^h\left(\frac94\right)^{-4h}.
\end{aligned}
\]
Since
$
        3\left(\frac94\right)^{-4}=3\left(\frac49\right)^4=\frac{768}{6561}<1,
$
the geometric sum is bounded by an absolute constant.  Hence
\[
        \Prob\bigl(\exists K<F_2:\ \rk K=2,\ \vol(K)<3L,
                  \ u_n,v_n\in K\bigr)
        \leq A_5 (L+1)^4 3^{3L}\left(\frac34\right)^{2n}.
\]
Since $L\leq n/7$, we have
$
        3^{3L}\left(\frac34\right)^{2n}
        \leq 3^{3n/7}\left(\frac34\right)^{2n}
        \leq 3^{-n/14},
$
because this is equivalent to
$
        3^{1/2}\leq \left(\frac43\right)^2,
$
and indeed $243<256$.  Therefore
\[
        \Prob\bigl(\exists K<F_2:\ \rk K=2,\ \vol(K)<3L,
                  \ u_n,v_n\in K\bigr)
        \leq A_6 (n+1)^4 3^{-n/14}.
\]
Combining this with the estimate for $\mathcal G_n^c$ gives
$
        \Prob\bigl(\varphi_{u_n,v_n}\text{ is not prime}\bigr)
        \leq A n^B 3^{-n/14}
$
for some absolute constants $A,B>0$.

Finally, on $\mathcal G_n$, the pair $(u_n,v_n)$ freely generates $H_n$.
Hence the endomorphism with $a\mapsto u_n$ and $b\mapsto v_n$ is injective.  By
Lemma~\ref{L.primecriterion}, absence of a proper intermediate rank-two
subgroup is exactly primeness.
\end{proof}

\begin{corollary}\label{C.mixed-spheres}\label{cor:mixed-sphere}
There are absolute constants $A,B>0$ such that for every pair of integers
$r,s\geq 1$, if $U_r$ and $V_s$ are independent uniformly distributed elements
of $S_r$ and $S_s$, and if $N=\min\{r,s\}$, then
\[
        \Prob\bigl(\varphi_{U_r,V_s}\text{ is not prime}\bigr)
        \leq A(r+s)^B 3^{-N/14}.
\]
\end{corollary}

\begin{proof}
Put
$
        N=\min\{r,s\},$ and $
        L=\left\lfloor\frac N7\right\rfloor.
$
Let $\widehat U_r$ and
$\widehat V_s$ be the cyclic reductions of $U_r$ and $V_s$.  Applying the cyclic
cancellation lemma separately to $U_r$ and $V_s$ gives
\[
        \Prob\bigl(|\widehat U_r|<6L\text{ or }|\widehat V_s|<6L\bigr)
        \leq A_0\left(3^{-(r-6L)/2}+3^{-(s-6L)/2}\right)
        \leq A_1 3^{-N/14}.
\]
Conditioned on $|\widehat U_r|=m$ and $|\widehat V_s|=\ell$, with $m,\ell\geq 6L$,
the cyclic reductions are independent uniformly distributed cyclically reduced
words of lengths $m$ and $\ell$.  The four cyclic words
$\widehat U_r,\widehat U_r^{-1},\widehat V_s,\widehat V_s^{-1}$ have at most
$2m+2\ell\leq 2(r+s)$ cyclic starting positions in total.  Therefore, by the
fixed cyclic block estimate,
\begin{multline*}
        \Prob\bigl((U_r,V_s)\text{ is not }L\text{-clean}
        \mid |\widehat U_r|=m,\ |\widehat V_s|=\ell\bigr) \\
        \leq A_2 (r+s)^2 3^{-L/2}
        \leq A_3 (r+s)^2 3^{-N/14}.
\end{multline*}
On the good event, the pair freely generates a rank-two subgroup.  If the
associated endomorphism is not prime, then both words lie in a proper rank-two
subgroup $K$ with $\vol(K)<3L$, by the deterministic quotient lemma.  The
counting lemma and the return estimate therefore give
\[
\begin{aligned}
        \Prob\bigl(\exists K<F_2:\ U_r,V_s\in K,\ \vol(K)<3L\bigr)
        &\leq A_4 (L+1)^4 3^{3L} \left(\frac34\right)^{r+s}
        \sum_{h\geq 0} 3^h\left(\frac94\right)^{-4h} \\
        &\leq A_5 (N+1)^4 3^{3L} \left(\frac34\right)^{2N} \\
        &\leq A_6 (N+1)^4 3^{-N/14}.
\end{aligned}
\]
Combining the three estimates and using $(N+1)^4\leq (r+s)^4$ proves the claim.
\end{proof}

\section{Bounds on the logarithmic density}\label{S.log-density}

\subsection{Upper bounds}

The sphere estimate from Theorem~\ref{T.sphere} is the technical result needed to prove the upper bound in Theorem~A.
We now transfer it to balls and compare it with explicit lower bounds, first
from fixed proper rank-two subgroups and then from a variable-overgroup
construction.

\begin{corollary}\label{C.ball.upper}
There are absolute constants $A,B>0$ such that
\[
        \Prob_{(u,v)\in B_n\times B_n}
        \bigl(\varphi_{u,v}\text{ is not prime}\bigr)
        \leq A n^B3^{-n/14}.
\]
\end{corollary}

\begin{proof}
The cases where $u=1$ or $v=1$ contribute $O(3^{-n})$, which is harmless.
For $r,s\geq 1$, condition on $|u|=r$ and $|v|=s$.  Since
$
        |S_r|=4\cdot 3^{r-1}$ and $
        |B_n|=2\cdot 3^n-1,
$
the contribution of lengths $r,s$ is bounded by
\[
        C 3^{r+s-2n}\cdot A(r+s)^B3^{-\min\{r,s\}/14}.
\]
Assume $r\leq s$.  Then
\[
        3^{r+s-2n}3^{-r/14}
        \leq 3^{r+n-2n}3^{-r/14}
        =3^{-n}3^{13r/14}
        \leq 3^{-n/14}.
\]
The same bound holds when $s\leq r$.  Summing over the at most $(n+1)^2$ pairs
of lengths contributes only a polynomial factor, so the whole probability is at
most $A'n^{B'}3^{-n/14}$ for suitable absolute constants $A',B'$.  Renaming the
constants gives the claim.
\end{proof}

For each $n\geq 0$, let
$
        \mathcal N_n=
        \{(u,v)\in B_n\times B_n:\varphi_{u,v}\text{ is not prime}\}
$
and define
\[
        \delta_{\mathrm{np}}=
        \limsup_{n\to\infty}\frac{\log |\mathcal N_n|}{\log |B_n\times B_n|}.
\]
Equivalently, if
$
        p_n=|\mathcal N_n| |B_n\times B_n|^{-1},
$
then
\[
        \delta_{\mathrm{np}}=
        1+\frac12\limsup_{n\to\infty}\frac{\log p_n}{n\log 3},
\]
because
$
        \log |B_n\times B_n|=2n\log 3+O(1).
$
In particular, Corollary~\ref{C.ball.upper} yields
\[
        \delta_{\mathrm{np}}\leq 1-\frac1{28}=\frac{27}{28}
\]
and this finishes the proof of the upper bound in Theorem~A.

\subsection{Lower bounds}

We now give lower bounds for $\delta_{\mathrm{np}}$. The first one is based on the idea of fixing an overgroup of rank two, which corresponds to fixing the first factor in the prime factorization, and the second one on fixing the second factor in the factorization. The second approach turns out to be more powerful, but we include the first one for completeness and because it is more elementary.

\begin{proposition}\label{P.fixedK.lower}
Let $K<F_2$ be a proper rank-two subgroup, and let
\[
        \rho(K)=\limsup_{n\to\infty}|K\cap B_n|^{1/n}
\]
be its ambient exponential growth rate.  Then
$
        \delta_{\mathrm{np}}
        \geq \frac{\log \rho(K)}{\log 3}.
$
\end{proposition}

\begin{proof}
Choose a proper finite-index subgroup $L<K$.  Since $K$ has rank two, such an
$L$ exists.  Because $K$ is a finite union of right cosets of $L$, and
multiplication by a fixed coset representative changes word length by only a
bounded amount, the sets $L\cap B_n$ and $K\cap B_n$ have the same exponential
growth rate:
\[
        \limsup_{n\to\infty}|L\cap B_n|^{1/n}=\rho(K).
\]

If $u,v\in L$ do not commute, then $\varphi_{u,v}$ is injective by
Lemma~\ref{L.injective.noncommuting}, and
$
        \langle u,v\rangle\leq L<K<F_2.
$
Hence $\varphi_{u,v}$ is not prime by Lemma~\ref{L.primecriterion}.

It remains to note that commuting pairs in $(L\cap B_n)^2$ are negligible on
the exponential scale.  For each nontrivial $u\in F_2$, the centralizer
$C_{F_2}(u)$ is cyclic, so
$
        |C_{F_2}(u)\cap B_n|=O(n).
$
Hence the number of commuting pairs in $(L\cap B_n)^2$ is at most
$O(n)|L\cap B_n|+O(|L\cap B_n|)$, which has exponential growth rate $\rho(K)$,
whereas $(L\cap B_n)^2$ has exponential growth rate $\rho(K)^2$.  Thus the
noncommuting pairs in $(L\cap B_n)^2$ have exponential growth rate $\rho(K)^2$.

Consequently
$
        |\mathcal N_n|\geq \exp\bigl((2\log \rho(K)+o(1))n\bigr).
$
Since
$
        |B_n\times B_n|=\exp\bigl((2\log 3+o(1))n\bigr),
$
we obtain
$
        \delta_{\mathrm{np}}\geq \frac{\log \rho(K)}{\log 3}.
$
\end{proof}

\begin{proposition}\label{P.a2b.lower}
Let
$
        K=\langle a^2,b\rangle<F(a,b).
$
Then
$
        \rho(K)=\rho_0,
$
where $\rho_0$ is the largest real root of
$
        \rho^3-\rho^2-\rho-3=0.
$
Explicitly,
\[
        \rho_0=
        \frac{1+\sqrt[3]{46+6\sqrt{57}}+\sqrt[3]{46-6\sqrt{57}}}{3}
        =2.1303954347\ldots .
\]
Consequently,
$
        \delta_{\mathrm{np}}
        \geq \log_3\rho_0=0.6884208564\ldots .
$
\end{proposition}

\begin{proof}
The Stallings core of $K$ is a figure-eight with one loop of length $2$,
labelled by $a^2$, and one loop of length $1$, labelled by $b$.  Let
$r=\rho^{-1}$.  Recall, by symmetry, the weighted non-backtracking transfer matrix
reduces to
\[
        \begin{pmatrix}
                r^2 & 2r \\
                2r^2 & r
        \end{pmatrix}.
\]
The critical value is determined by the condition that this matrix have
spectral radius $1$, equivalently
$
        (1-r^2)(1-r)-4r^3=0.
$
Thus
$
        1-r-r^2-3r^3=0.
$
Putting $\rho=r^{-1}$, this becomes
$
        \rho^3-\rho^2-\rho-3=0.
$
As before, Cardano's formula gives the displayed expression for $\rho_0$.  The lower bound
now follows from Proposition~\ref{P.fixedK.lower}.
\end{proof}

Thus Proposition~\ref{P.fixedK.lower} cannot yield a lower bound larger than
$\log_3\rho_0$.  The next proposition shows that this fixed-overgroup mechanism
is not the dominant one.

\begin{proposition}\label{P.variable.lower}
One has
$
        \delta_{\mathrm{np}}\geq \frac34.
$
\end{proposition}

\begin{proof}
Let
$
        \Gamma=\ker\bigl(F(a,b)\to (\mathbb Z/2\mathbb Z)^2\bigr)
$
be the kernel of the homomorphism sending $a$ and $b$ to the two standard
generators.  Then $\Gamma$ is a proper finite-index subgroup of $F(a,b)$, and
therefore
$
        |\Gamma\cap B_r|=3^{r+o(r)}.
$

Put $m=\lfloor n/2\rfloor$.  For $x\in \Gamma\cap B_m$ and
$y\in \Gamma\cap B_n$ with $[x,y]\neq 1$, consider the endomorphism
\[
        \psi_{x,y}(a)=x^2,
        \qquad
        \psi_{x,y}(b)=y.
\]
Since $|x^2|\leq 2|x|\leq n$ and $|y|\leq n$, this gives a pair in
$B_n\times B_n$.

Let
\[
        \varphi_{a^2,b}(a)=a^2,
        \qquad
        \varphi_{a^2,b}(b)=b,
\]
and let
\[
        \alpha(a)=x,
        \qquad
        \alpha(b)=y.
\]
Then $\psi_{x,y}=\alpha\circ\varphi_{a^2,b}$.  Since $a^2$ and $b$ do not commute,
Lemma~\ref{L.injective.noncommuting} shows that $\varphi_{a^2,b}$ is injective.  Its
image lies in the proper subgroup $\langle a^2,b\rangle<F(a,b)$, so $\varphi_{a^2,b}$ is
not an automorphism.  Since $x$ and $y$ do not commute, Lemma~\ref{L.injective.noncommuting}
also shows that $\alpha$ is injective.  Moreover
$\alpha(F(a,b))\subseteq \Gamma<F(a,b)$, so $\alpha$ is not an automorphism.
Thus $\psi_{x,y}$ is a non-prime injective endomorphism.

It remains to count.  The pairs with $x=1$ contribute only $3^{n+o(n)}$.
For each nontrivial $x\in \Gamma\cap B_m$, the centralizer $C_{F_2}(x)$ is
cyclic, so
$
        |C_{F_2}(x)\cap B_n|=O(n).
$
Hence the commuting pairs in $(\Gamma\cap B_m)\times(\Gamma\cap B_n)$ are
exponentially negligible, and the number of noncommuting such pairs is
$
        3^{m+n+o(n)}=3^{3n/2+o(n)}.
$
The squaring map is injective in a free group, since free groups have unique
roots.  Therefore these choices give $3^{3n/2+o(n)}$ distinct non-prime pairs
in $B_n\times B_n$.

Since
$
        |B_n\times B_n|=3^{2n+o(n)},
$
we obtain
$
        \delta_{\mathrm{np}}\geq \frac{3/2}{2}=\frac34.
$
\end{proof}

The preceding estimates leave a substantial gap:
\[
        \frac34
        \leq \delta_{\mathrm{np}}
        \leq \frac{27}{28}=0.9642857142\ldots .
\]
It would be interesting to determine the exact value of $\delta_{\mathrm{np}}$, as suggested by the conjecture in the introduction.

\medskip

The whole situation differs substantially from classical prime-number-type sieving. Indeed, while the conditions
$u,v\in K$, as $K$ ranges over proper rank-two subgroups, have large overlaps
and are organized by non-distributive overgroup lattices, the
inclusion-exclusion is unlikely to be useful with respect to logarithmic density without
additional structural input.

\subsection{The simple random-walk model}\label{S.random-walk}

Let $\mu$ be the uniform probability measure on $\{a^{\pm1},b^{\pm1}\}$.  For
$n\geq 0$, let $w_n=s_1s_2\cdots s_n$
be the simple random walk on $F_2$, followed by free reduction.  Let
$w_n^{(1)},w_n^{(2)}$ be two independent copies and put
$\mathcal H_n=\langle w_n^{(1)},w_n^{(2)}\rangle$.

\begin{lemma}\label{L.radial.exp}
Let $w_n$ be simple random walk on $F_2$, freely reduced after $n$ steps, and
put $R_n=|w_n|$.  Then, for every $\lambda\geq 0$,
\[
        \E e^{-\lambda R_n}
        \leq
        \left(\frac34 e^{-\lambda}+\frac14 e^{\lambda}\right)^n.
\]
Consequently, if $R_n^{(1)},R_n^{(2)}$ are independent copies of $R_n$, then
\[
        \E e^{-\lambda \min\{R_n^{(1)},R_n^{(2)}\}}
        \leq
        2\left(\frac34 e^{-\lambda}+\frac14 e^{\lambda}\right)^n.
\]
\end{lemma}

\begin{proof}
Let $T_n=X_1+\cdots+X_n$, where the $X_i$ are independent and satisfy
\[
        \Prob(X_i=1)=\frac34,
        \qquad
        \Prob(X_i=-1)=\frac14.
\]
The radial process $R_n=|w_n|$ moves up with probability $3/4$ and down with
probability $1/4$ away from the identity, while at the identity it is forced to
move up.  Hence one may couple $R_n$ and $T_n$ so that
$
        R_n\geq T_n
$
for all $n$.  Since $\lambda\geq 0$, this gives
\[
        \E e^{-\lambda R_n}
        \leq
        \E e^{-\lambda T_n}
        =
        \left(\frac34 e^{-\lambda}+\frac14 e^{\lambda}\right)^n.
\]
For the second assertion, use
$
        e^{-\lambda \min\{R_n^{(1)},R_n^{(2)}\}}
        \leq
        e^{-\lambda R_n^{(1)}}+e^{-\lambda R_n^{(2)}}.
$
\end{proof}

\begin{lemma}\label{L.radial.uniform}
Conditioned on $|w_n|=r$, 
the reduced word $w_n$ is uniformly distributed on the sphere $S_r$.
\end{lemma}

\begin{proof}
The simple random walk measure on $F_2$ is invariant under every automorphism of
the labelled Cayley tree fixing the identity.  This automorphism group acts transitively on each sphere $S_r$.
Therefore the conditional distribution on $S_r$ is uniform.
\end{proof}

\begin{theorem}\label{T.random-walk}
There are absolute constants $A,B>0$ such that
\[
        \Prob\bigl(\varphi_{w_n^{(1)},w_n^{(2)}}\text{ is not prime}\bigr)
        \leq A n^B 3^{-n/30}.
\]
\end{theorem}

\begin{proof}
Put
$
        R_i=|w_n^{(i)}|,$ for $i=1,2$.
Conditioned on $R_1=r$ and $R_2=s$, with $r,s\geq 1$, the two reduced words
$w_n^{(1)}$ and $w_n^{(2)}$ are independent and uniformly distributed on
$S_r$ and $S_s$, respectively, by Lemma~\ref{L.radial.uniform}.

By Corollary~\ref{C.mixed-spheres}, there are absolute constants $A_0,B>0$
such that
\[
\begin{aligned}
        \Prob\bigl(&\varphi_{w_n^{(1)},w_n^{(2)}}\text{ is not prime}
        \,\big|\, R_1=r,\ R_2=s\bigr)
        \leq A_0 (r+s)^B 3^{-\min\{r,s\}/14}.
\end{aligned}
\]
Since $r+s\leq 2n$, averaging gives
\[
\begin{aligned}
&\Prob\bigl(\varphi_{w_n^{(1)},w_n^{(2)}}\text{ is not prime},\ R_1,R_2\geq 1\bigr)
\leq A_1 n^B \E\,3^{-\min\{R_1,R_2\}/14}.
\end{aligned}
\]
Apply Lemma~\ref{L.radial.exp} with $\lambda=(\log 3)/14$.  Then
\[
        \E\left(3^{-\min\{R_1,R_2\}/14}\right)
        \leq
        2\left(\frac34\,3^{-1/14}+\frac14\,3^{1/14}\right)^n.
\]
The numerical inequality
$
        \frac34\,3^{-1/14}+\frac14\,3^{1/14}
        <
        3^{-1/30}
$
therefore implies
\[
        \Prob\bigl(\varphi_{w_n^{(1)},w_n^{(2)}}\text{ is not prime},\ R_1,R_2\geq 1\bigr)
        \leq
        A_2 n^B 3^{-n/30}.
\]

It remains to consider the case $R_1=0$ or $R_2=0$.  Since
$
        \mathbf 1_{\{R_i=0\}}\leq 3^{-R_i/14},
$
the same exponential moment estimate gives
\[
        \Prob(R_i=0)
        \leq
        \left(\frac34\,3^{-1/14}+\frac14\,3^{1/14}\right)^n
        \leq
        3^{-n/30}.
\]
Thus
$
        \Prob(R_1=0\text{ or }R_2=0)
        \leq
        2 \cdot 3^{-n/30}.
$
Absorbing this term into the polynomial prefactor proves
$
        \Prob\bigl(\varphi_{w_n^{(1)},w_n^{(2)}}\text{ is not prime}\bigr)
        \leq A n^B 3^{-n/30}.
$
This proves the theorem.
\end{proof}

\section{Questions and remarks}\label{S.questions}

Despite the examples of factorization lattices appearing in Section \ref{sec:examples}, we believe that \emph{generically}, the factorization lattice of a non-prime injective endomorphism of $F_2$ is a chain. There are various potential ways to make this precise: Let $k \in \mathbb N$ be fixed and consider the set $(B_n \times B_n)^k$ with the uniform probability measure. Any element $((u_1,v_1),\ldots,(u_k,v_k)) \in (B_n \times B_n)^k$ defines a $k$-tuple of endomorphisms $(\varphi_{u_i,v_i})_{i=1}^k$, and we can consider the event $\Sigma(n,k)$ that the factorization lattice of the endomorphism $\varphi_\xi = \varphi_{u_1,v_1} \circ \cdots \circ \varphi_{u_k,v_k}$ consists just of the obvious chain of length $k+1$.

\begin{question}
    \label{Q.generic-chain} Let $k \in \mathbb N$ be fixed.
    Is it true that the probability of $\Sigma(n,k)$ tends to $1$ as $n \to \infty$?
\end{question}

Apart from the generic behaviour it remains an interesting problem to understand the possible factorization lattices of non-prime injective endomorphisms of $F_2$, see for example Question \ref{Q:lattice}. Moreover, it seems interesting to determine the maximal size of a factorization lattice in terms of the volume of the Stallings core. The bounds from Lemma \ref{L.count-cores} seem to be far from optimal in that regard and examples suggest the following question:

\begin{question}
Is the size of $P_2(K)$ bounded by $O(\vol(K) \log \log \vol(K))$?
\end{question}

This bound is achieved by the example of $K=\langle a^n,b a^n b^{-1}\rangle$ for some large $n$, which has volume $2n+1$ and $P_2(K)$ of size $\geq \sigma(n)$, where $\sigma(n)$ is the sum of divisors of $n$. Indeed, for $n:={\rm lcm}(1,2,\ldots,m)$ and $m$ large, we have $\sigma(n) \approx n \log\log(n)$.

\section*{Acknowledgements}

I thank Vadim Alekseev for interesting discussions related to this problem. I am also indebted to Sean Eberhard, who generously shared his insight into the use of Stallings graphs in this context.

\medskip

ChatGPT 5.5 was used to assist in drafting parts of this manuscript. All content was reviewed and substantially revised by the author, who is responsible for the final text.

\end{document}